\documentclass[a4paper,11pt]{article}

\usepackage[a4paper,margin=1.5cm]{geometry}
\usepackage{amsthm,amsmath,amssymb}
\usepackage{mathtools}
\usepackage{hyperref}

\theoremstyle{plain}
\newtheorem{theorem}{Theorem}[section]
\newtheorem{prop}[theorem]{Proposition}
\newtheorem{cor}[theorem]{Corollary}
\newtheorem{lemma}[theorem]{Lemma}

\theoremstyle{definition}
\newtheorem{defi}[theorem]{Definition}

\theoremstyle{remark}

\newtheorem{ex}[theorem]{Example}

\newcommand{\Cij}{\mathcal{C}_{ij}}
\newcommand{\Dij}{\mathcal{D}_{ij}}
\newcommand{\Eij}{\mathcal{E}_{ij}}
\newcommand{\Fp}{\mathbb{F}_p}
\newcommand{\Fq}{\mathbb{F}_q}
\newcommand{\RR}{\mathbb{R}}
\newcommand{\ZZ}{\mathbb{Z}}
\newcommand{\PP}{\mathbb{P}}
\newcommand{\im}{{\rm im}}
\newcommand{\coker}{{\rm coker}}

\title{Bivariate trinomials over finite fields}
\author{M.~Avenda\~no and J.~Mart\'\i n-Morales\footnote{Partially supported by
the Spanish Government MTM2016-76868-C2-2-P,
\emph{Grupo E15 Gobierno de Arag\'on/Fondo Social Europeo},
and FQM-333 from \emph{Junta de Andaluc{\'\i}a}.}}
\date{\today}

\newcommand{\Addresses}{{
  \bigskip
  \footnotesize

  M.~Avendano, \textsc{Centro Universitario de la Defensa, Academia General Militar,
  Ctra.~de Huesca s/n., 50090, Zaragoza, Spain}\par\nopagebreak
  \textit{E-mail address}: \texttt{avendano@unizar.es}

  \medskip

  J.~Mart\'{\i}n-Morales, \textsc{Centro Universitario de la Defensa, Academia General Militar,
  Ctra.~de Huesca s/n., 50090, Zaragoza, Spain}\par\nopagebreak
  \textit{E-mail address}: \texttt{jorge@unizar.es}

}}

\begin{document}

\maketitle

\begin{abstract}
We study the number of points in the family of plane curves defined by a trinomial
\[
  \mathcal{C}(\alpha,\beta)=
     \{(x,y)\in\Fq^2\,:\,\alpha x^{a_{11}}y^{a_{12}}+\beta x^{a_{21}}y^{a_{22}}=x^{a_{31}}y^{a_{32}}\}
\]
with fixed exponents (not collinear) and varying coefficients over finite fields. We prove that
each of these curves has an almost predictable number of points, given by a closed formula that
depends on the coefficients, exponents, and the field, with a small error term $N(\alpha,\beta)$
that is bounded in absolute value by $2\tilde{g}q^{1/2}$, where $\tilde{g}$ is a constant that
depends only on the exponents and the field. A formula for $\tilde{g}$ is provided, as well as a comparison of $\tilde{g}$ with the genus $g$ of the projective closure of the curve over
$\overline{\Fq}$. We also give several linear and quadratic identities for the numbers
$N(\alpha,\beta)$ that are strong enough to prove the estimate above, and in some cases, to
characterize them completely.
\end{abstract}

\section{Introduction}

The main result in this article is inspired by Theorem~\ref{thm_gauss} given below, proven by Gauss
in his book Disquisitiones Arithmeticae~\cite[Thm.~358]{Gauss}. We have used a mildly rephrased version
of the original theorem, taken from~\cite[page~111]{ST92}, that better matches our more modern notation.

\begin{theorem}[Gauss]\label{thm_gauss}
Let $p$ be an odd prime and let $M_p$ be the number of points in the projective curve
$\{[x:y:z]\in\mathbb{P}^2(\Fp)\,:\,x^3+y^3+z^3=0\}$.
\begin{enumerate}
\item If $p\not\equiv 1\pmod{3}$, then $M_p=p+1$.
\item If $p\equiv 1\pmod{3}$, then the equation $u^2+27\bar{v}^2=4p$ has a unique integer solution (up
 to the signs), and if $u$ is chosen such that $u\equiv 1\pmod{3}$, then $M_p=p+1+u$.
\end{enumerate}
\end{theorem}

In a few words, Gauss' theorem says that the number of (projective) points in the plane curve
$x^3+y^3+z^3\equiv0\pmod{p}$ is $p+1$ plus a small error term $u$ (that only appears when
$p\equiv1\pmod{3}$) which is characterized by the quadratic equation $u^2+27\bar{v}^2=4p$ with integral
unknowns. Our main result (Thm.~\ref{mainth}) is a generalization of Gauss' theorem for any
non-degenerate trinomial equation in two variables, over any finite field, where we show that the
number of points is a predictable number (given by a closed formula in terms of the coefficients,
exponents, and the field) plus an error term which also satisfies an explicit quadratic equation in
many unknowns, all of them having a precise meaning (as opposed to Gauss' theorem, where only the
variable $u$ matters). 
More precisely, our result gives, in the case $p\equiv1\pmod{3}$, that $M_p=p+1+u$, where
$u^2+v^2+uv=3p$ for some $u,v\in\ZZ$. The symmetry of the curve allows one to rewrite it as
$u^2+27\bar{v}^2=4p$, where $\bar{v}=\frac{2v+u}{9}\in\ZZ$ and to show that $u\equiv1\pmod{3}$.
All the details are given in Section~\ref{gauss}.

Note that Gauss' theorem implies that the error term $u$ is bounded in absolute value by $2\sqrt{p}$.
This observation was generalized by Hasse to elliptic curves over finite fields~\cite[Ch.~5, Thm.~1.1]{Si86},
then by Weil to hypersurfaces defined by an equation of the type
$\alpha_0x_0^{a_0}+\alpha_1x_1^{a_1}+\dots+\alpha_rx_r^{a_r}=b$~\cite{Weil49}, which led to the statement of
the famous Weil's conjectures, finally proven by Dwork~\cite{Dwork}, Grothendieck~\cite{Groth}, and Deligne~\cite{Deligne}
for any smooth hypersurface.

With our approach, the estimate of the error follows from a simple computation using Lagrange multipliers (see
Prop.~\ref{lagrange}).
In contrast with the results above, our proof is elementary and the estimate is valid for any trinomial
(not necessarily smooth). Moreover, our estimate $2\tilde{g}q^{1/2}$ (see Cor.~\ref{maincor}) is better that the bound
obtained from Weil's conjectures $2gq^{1/2}$, since the genus $g$ is an invariant that only reflects the complex geometry
of the curve, while our $\tilde{g}$ includes also information about the field. In Section~\ref{sec4}, we obtain a closed formula for the genus $g$ of a trinomial plane curve (see Prop.~\ref{prop-genus}),
that can be compared term by term with the definition of $\tilde{g}$ given in~\eqref{eqdef}. For instance, in the case of Gauss' theorem,
the curve has genus $g=1$, but our $\tilde{g}$ is zero when $p\not\equiv1\pmod{3}$, hence capturing both cases
of the statement in a unified way.

A bound for trinomials (of the same type studied by Weil), that closely resembles ours, was obtained by Hua and
Vandiver~\cite{HV49}. However, their result follows from estimates using characters, while ours is a consequence
of a quadratic optimization problem over $\RR$. Some experiments show that a much better estimate could be computed
if we were able to solve the optimization problem over the integers (see Example~\ref{example3}).

\section{Statement of the results}

Let $p$ be a prime and $q=p^n$ for some $n\geq 1$. Let $\rho$ be a generator of the cyclic group $\Fq^*$. Consider the curve
\[ \Cij = \{ (x,y)\in \Fq^2\,:\, \rho^i x^{a_{11}}y^{a_{12}} + \rho^j x^{a_{21}}y^{a_{22}} = x^{a_{31}}y^{a_{32}} \},\]
and let $\Cij^*=\Cij\cap(\Fq^*)^2$.

To avoid a degenerate case, we assume that the exponents vectors $(a_{11},a_{12})$, $(a_{21},a_{22})$, $(a_{31},a_{32})$ are not collinear, i.e.~the matrix $B = \begin{bmatrix} b_{11} & b_{12} \\ b_{21} & b_{22}\end{bmatrix}
:=\begin{bmatrix} a_{11}-a_{31} & a_{12}-a_{32} \\ a_{21}-a_{31} & a_{22}-a_{32}\end{bmatrix}$ is invertible. 

We need the following constants derived from $B$:
\begin{equation}\label{eqdef}
\begin{aligned}
 d &= \gcd(b_{11},b_{12},q-1), \\
 e &= \gcd(b_{21},b_{22},q-1), \\
 f &= \gcd(b_{11}-b_{21}, b_{12}-b_{22},q-1), \\
 k &= \gcd((q-1)\gcd(d,e,f),\det(B)), \\
 w &= \begin{cases} 0 & q \text{ even}, \\ \frac{q-1}{2} & q \text{ odd}, \end{cases} \\
 \tilde{g} &= \frac{1}{2} ( k-d-e-f+2 ).
\end{aligned}
\end{equation}

The value $k$ corresponds to $|\coker(B)|$,  where $B$ is regarded as a group homomorphism
$B:\ZZ_{q-1}^2\to\ZZ_{q-1}^2$ given by the multiplication $v\mapsto Bv$ (see Lemma~\ref{orden}).

Our goal is to estimate the number of points $|\Cij|$ and $|\Cij^*|$ for all $i,j$.
Since $\rho^{q-1}=1$, the indices $i$ and $j$ can be regarded modulo $q-1$.

\begin{defi}
 $D_{\ell}(i)=\begin{cases} \ell & \text{if}\;\ell\,|\,i, \\ 0 & \text{otherwise.}\end{cases}$
\end{defi}

Note that $|\Cij| = |\Cij^*|+|\Cij\cap\{x=0,y\neq0\}|+|\Cij\cap\{y=0,x\neq0\}|+|\Cij\cap\{x=y=0\}|$,
and that the points in $\Cij\cap\{x=0,y\neq0\}$ and $\Cij\cap\{y=0,x\neq0\}$ correspond to the solutions in $\Fq^*$ of a univariate equation with at most two non-zero terms. Therefore, $|\Cij\cap\{x=0,y\neq0\}|$ and $|\Cij\cap\{y=0,x\neq0\}|$ can be computed exactly with a closed formula in terms of $i$, $j$, $q$,
and the exponents (see Lemma~\ref{lemaD2}). Moreover, $|\Cij\cap\{x=y=0\}|$ is either $1$ or $0$,
depending on whether $a_{11}+a_{12}$, $a_{21}+a_{22}$, and $a_{31}+a_{32}$ are all positive
or not. This means that $|\Cij|$ and $|\Cij^*|$ can be easily derived from each other. For this reason,
and to avoid discussing several cases depending on the configuration of the exponents, we present our
results only for $|\Cij^*|$, which can be done with a more uniform notation.

\begin{theorem}\label{mainth}
With the notation given above, we have
\begin{equation}\label{eq:defCij}
|\Cij^*| = q+1-D_d(i)-D_e(j)-D_f(i-j+w)+N_{ij}
\end{equation}
for some integers $N_{ij}$ that satisfy:
\begin{enumerate}
 \item \label{x2} $\displaystyle\sum_{j=0}^{q-2}N_{ij}=0$ for all $i$,
 \item \label{x1} $\displaystyle\sum_{i=0}^{q-2}N_{ij}=0$ for all $j$,
 \item \label{x5} $\displaystyle\sum_{i-j=r}N_{ij}=0$ for all $r$,
 \item \label{x4} $\displaystyle N_{i+b_{11},j+b_{21}}=N_{ij}=N_{i+b_{12},j+b_{22}}$ for all $i,j$,
 \item \label{x3} $\displaystyle\sum_{i=0}^{q-2}\sum_{j=0}^{q-2}N_{ij}^2 = 2 \tilde{g} (q-1)^2 q = (q-1)^2q(k-d-e-f+2)$.
\end{enumerate}
\end{theorem}

Using \eqref{x4}, the sum of Theorem~\ref{mainth}\eqref{x3} can be rewritten taking only one representative
of each $(i,j)$ modulo the subgroup $\langle (b_{11},b_{12}), (b_{21},b_{22})\rangle\subseteq \ZZ_{q-1}^2$,
\begin{equation}\label{bound1}
 \sum_{\overline{(i,j)}\in\coker(B)}\!\!\!\!\!\!N_{ij}^2 = 2 \tilde{g} k q = kq(k-d-e-f+2)\leq k^2q.
\end{equation}
We immediately obtain the upper bound $|N_{ij}|\leq k\sqrt{q}$ for all $i,j$. Using a similar approach, but taking advantage
of~\eqref{x2},~\eqref{x1}, and~\eqref{x5}, it is possible to get a stronger upper bound:

\begin{cor}\label{maincor}
$|N_{ij}|\leq 2 \tilde{g} \sqrt{q}$ for all $i,j$.
\end{cor}

\section{Proof of the main results}

\begin{lemma}\label{lemaD} For any $r\geq 1$,
 \[ \sum_{i=0}^{q-2}D_{\ell}(i)^r=\ell^{r-1}(q-1). \]
\end{lemma}
\begin{proof}
 By definition of $D_{\ell}$ we have:
 \[ \sum_{i=0}^{q-2}D_{\ell}(i)^r=\sum_{\ell|i}\ell^r=\ell^r\cdot\frac{q-1}{\ell}=\ell^{r-1}(q-1), \]
 since the number of indices $0\leq i<q-1$ that are divisible by $\ell$ is exactly $\frac{q-1}{\ell}$.
\end{proof}

\begin{lemma}\label{lemaD2} For any $a_1,\ldots,a_m\in\ZZ$,
 \[  \left|\left\{(x_1,\ldots,x_m)\in(\Fq^*)^m\,:\,\rho^ix_1^{a_1}\cdots x_m^{a_m}=1\right\}\right| = (q-1)^{m-1}D_{\ell}(i), \]
 where $\ell=\gcd(a_1,\ldots,a_m,q-1)$.
\end{lemma}
\begin{proof}
 Consider the group homomorphism $\varphi:(\Fq^*)^m\to \Fq^*$ given by $(x_1,\ldots,x_m)\mapsto x_1^{a_1}\cdots x_m^{a_m}$.
 The image of $\varphi$ is generated by $\rho^{a_1},\ldots,\rho^{a_m}$, which is also generated by $\rho^\ell$ since the
 group $\Fq^*$ is cyclic, and in particular $|\im(\varphi)|=\frac{q-1}{\ell}$. When $\rho^{-i}\not\in\langle\rho^\ell\rangle$,
 i.e. $\ell\nmid i$, the left-hand side and the right-hand side of the equation in the statement are both clearly zero.
 Otherwise, when $\ell\,|\,i$, the number of solutions is equal to $|\coker(\varphi)|=(q-1)^m/|\im(\varphi)|=(q-1)^{m-1}\ell=(q-1)^{m-1}D_\ell(i)$.
\end{proof}

\begin{lemma}\label{orden}
We have
\begin{enumerate}
\item \label{orden1} $|\coker(B)|=k$.
\item \label{orden2} The subgroups $\langle\overline{(1,0)}\rangle$, $\langle\overline{(0,1)}\rangle$, $\langle\overline{(1,1)}\rangle$
       of $\coker(B)$ have orders $\frac{k}{e}$, $\frac{k}{d}$, $\frac{k}{f}$, respectively.
\end{enumerate}
\end{lemma}
\begin{proof}
 \eqref{orden1} Define the matrix $L=\left[\begin{array}{cccc} b_{11} & b_{12} & q-1 & 0 \\ b_{21} & b_{22} & 0 & q-1\end{array}\right]\in\ZZ^{2\times 4}$,
 which can be regarded as a linear map $L:\ZZ^4\to\ZZ^2$, whose cokernel is 
\[ 
  \coker(B) =\ZZ_{q-1}^2/\langle (b_{11},b_{12}), (b_{21},b_{22})\rangle \cong \ZZ^2/\im(L).
\]
Note that $|\ZZ^2/\im(L)|$ is invariant under elementary row or column operations (on $L$). Therefore, we can substitute $L$ by
its Smith Normal form, and in particular $|\ZZ^2/\im(L)|$ is equal to the greatest common divisor of the determinants of the
$2\times 2$ minors of $L$, i.e.
\[
 |\coker(B)| = |\ZZ^2/\im(L)| = \gcd(\det(B), (q-1)d, (q-1)e) = k.
\]
\eqref{orden2} It is enough to show that $|\langle\overline{(1,0)}\rangle|=k/e$, since the other two are analogous. By definition, the order
of $\overline{(1,0)}$ is
 \[
   \min\{ r\geq 1\,:\, (r,0)\in\im(L)\} = 
   \min\big\{r\geq 1\,:\,|\coker(L)|=|\coker(\left[L|\begin{smallmatrix} r\\ 0\end{smallmatrix}\right])|\big\}.
\]
The greatest common divisor of the determinant of the $2\times 2$ minors of the extended matrix $[L|\begin{smallmatrix} r\\ 0\end{smallmatrix}]$ that do not appear in $L$ is $\gcd(r(q-1), rb_{21}, rb_{22})=re$. Therefore, $|\langle\overline{(1,0)}\rangle|=\min\{r\geq1\,:\, k=\gcd(k,re)\}=k/e$.
\end{proof}

\begin{proof}[Proof of Theorem \ref{mainth}]
 We prove \eqref{x2}, since the proofs of \eqref{x1} and \eqref{x5} are analogous.
 Note that the sets $\Cij^*$ for $j=0,\ldots,q-2$ are disjoint, thus
 \[
   \begin{aligned}
      \sum_{j=0}^{q-2}|\Cij^*| & = \left|\bigcup_{j=0}^{q-2}\Cij^*\right| = |\{(x,y)\in(\Fq^*)^2\,:\, \rho^ix^{a_{11}-a_{31}}y^{a_{12}-a_{32}}\neq 1 \}| \\
      & = (q-1)^2-D_d(i)(q-1).
    \end{aligned}
\]
Therefore,
\[
 \begin{aligned}
 \sum_{j=0}^{q-2}N_{ij} &= \sum_{j=0}^{q-2}\left(|\Cij^*|+D_d(i)+D_e(j)+D_f(i-j+w)-(q+1)\right) \\
    &= (q-1)^2-D_d(i)(q-1) + D_d(i)(q-1) + \sum_{j=0}^{q-2}D_e(j) \\
    & \quad +\sum_{j=0}^{q-2}D_f(i-j+w) - (q+1)(q-1)
  \end{aligned}
\]
which is equal to zero by Lemma~\ref{lemaD}.

To prove~\eqref{x4}, note that the map $\mathcal{C}_{i+b_{11},j+b_{21}}^*\to\Cij^*$ given by
$(x,y)\mapsto(\rho x, y)$ is a bijection, so
$|\mathcal{C}_{i+b_{11},j+b_{21}}^*|=|\Cij^*|$. Moreover, $D_d(i+b_{11})=D_d(i)$, $D_e(j+b_{21})=D_e(j)$,
and $D_f(i-j+b_{11}-b_{21}+w)=D_f(i-j+w)$ since $d\,|\,b_{11}$, $e\,|\,b_{21}$, and $f\,|\,b_{11}-b_{21}$ by
definition. This implies that $N_{i+b_{11},j+b_{21}}=N_{ij}$. The proof of $N_{i+b_{12},j+b_{22}}=N_{ij}$ is analogous.

Now we prove \eqref{x3},
\[
\begin{aligned}
 \sum_{i,j}&|\Cij^*|^2 = \sum_{i,j}\left(N_{ij}-D_d(i)-D_e(j)-D_f(i-j+w)+q+1 \right)^2 = \\
   &= \sum_{i,j}N_{ij}^2 + \sum_{i,j}D_d(i)^2 + \sum_{i,j}D_e(j)^2 + \sum_{i,j}D_f(i-j+w)^2 + (q+1)^2(q-1)^2   \\
   &\qquad -2\sum_{i,j}N_{ij}D_d(i)-2\sum_{i,j}N_{ij}D_e(j)-2\sum_{i,j}N_{ij}D_f(i-j+w)+2(q+1)\sum_{i,j}N_{ij} \\
   &\qquad +2\sum_{i,j}D_d(i)D_e(j)+2\sum_{i,j}D_d(i)D_f(i-j+w)+2\sum_{i,j}D_e(j)D_f(i-j+w)                      \\
   &\qquad -2(q+1)\sum_{i,j}D_d(i)-2(q+1)\sum_{i,j}D_e(j)-2(q+1)\sum_{i,j}D_f(i-j+w).
\end{aligned}
\]
By \eqref{x2}, \eqref{x1}, and \eqref{x5} the sixth, seventh, eighth, and ninth terms vanish. The other
terms can be calculated by Lemma~\ref{lemaD}, thus
\begin{equation}\label{eqxxx}
  \sum_{i,j}|\Cij^*|^2 = \sum_{i,j}N_{ij}^2 + (q-1)^2(q^2-4q+1+d+e+f).
\end{equation}

Note that $|\Cij^*|^2 = |\Cij^*\times\Cij^*|$,
\[
  \Cij^*\times\Cij^* =
    \left\{(x_1,y_1,x_2,y_2)\in(\Fq^*)^4\,:\,\begin{array}{c} \rho^ix_1^{b_{11}}y_1^{b_{12}}+\rho^jx_1^{b_{21}}y_1^{b_{22}}=1 \\
    \rho^ix_2^{b_{11}}y_2^{b_{12}}+\rho^jx_2^{b_{21}}y_2^{b_{22}}=1 \end{array} \right\}.
\]
Define $\Delta=\det\begin{bmatrix}x_1^{b_{11}}y_1^{b_{12}} & x_1^{b_{21}}y_1^{b_{22}} \\ x_2^{b_{11}}y_2^{b_{12}} & x_2^{b_{21}}y_2^{b_{22}}\end{bmatrix}$. The set $\Cij^*\times\Cij^*$ can be written as the disjoint union $\Dij\cup\Eij$, where
$\Dij=(\Cij^*\times\Cij^*)\cap\{(x_1,y_1,x_2,y_2)\in(\Fq^*)^4\,:\,\Delta\neq 0\}$ and $\Eij=(\Cij^*\times\Cij^*)\cap\{(x_1,y_1,x_2,y_2)\in(\Fq^*)^4\,:\,\Delta= 0\}$.
By Cramer's rule,
\[
 \Dij=\left\{ (x_1,y_1,x_2,y_2)\in(\Fq^*)^4\,:\, \Delta\neq 0,\begin{array}{c} \rho^i = (-x_1^{b_{21}}y_1^{b_{22}}+ x_2^{b_{21}}y_2^{b_{22}})/\Delta \\ \rho^j = (x_1^{b_{11}}y_1^{b_{12}}-x_2^{b_{11}}y_2^{b_{12}})/\Delta \end{array}\right\},
\]
which imply that the $\Dij$ are disjoint and their union is
\[
 \bigcup_{i,j}\Dij=\left\{ (x_1,y_1,x_2,y_2)\in(\Fq^*)^4\,:\, \Delta\neq 0, \begin{array}{c}x_1^{b_{21}}y_1^{b_{22}}\neq x_2^{b_{21}}y_2^{b_{22}} \\ x_1^{b_{11}}y_1^{b_{12}}\neq x_2^{b_{11}}y_2^{b_{12}} \end{array}\right\}.
\]
Introducing the change of variables $x=x_1/x_2$ and $y=y_1/y_2$, we get
\[
\begin{aligned}
 &\sum_{ij}|\Dij| =\left|\bigcup_{i,j}\Dij\right|=(q-1)^2\left|\left\{(x,y)\in(\Fq^*)^2\,:\, \begin{array}{c} x^{b_{11}}y^{b_{12}}\neq 1 \\
  x^{b_{21}}y^{b_{22}}\neq 1 \\ x^{b_{11}}y^{b_{12}} \neq x^{b_{21}}y^{b_{22}}
\end{array}\right\}\right| \\
 & = (q-1)^2\bigg( (q-1)^2- \Big|\underbracket{\{x^{b_{11}}y^{b_{12}}=1 \}}_{S_1} \cup \underbracket{\{x^{b_{21}}y^{b_{22}}=1 \}}_{S_2} \cup \underbracket{\{x^{b_{11}-b_{21}}y^{b_{12}-b_{22}}=1 \}}_{S_3} \Big| \bigg).
\end{aligned}
\]
Note that $S_1\cap S_2=S_1\cap S_3=S_2\cap S_3=S_1\cap S_2\cap S_3$, so $|S_1\cup S_2\cup S_3|=|S_1|+|S_2|+|S_3|-2|S_1\cap S_2|$.
By Lemma~\ref{lemaD2}, $|S_1|=(q-1)d$, $|S_2|=(q-1)e$, and $|S_3|=(q-1)f$. Moreover, $|S_1\cap S_2|=|\coker(B)|=k$.
All together, we get
\[
 \sum_{i,j}|\Dij|=(q-1)^2\bigg[(q-1)^2-(q-1)(d+e+f)+2k\bigg].
\]
Observe that
\[
 \Eij=\left\{(x_1,y_1,x_2,y_2)\in(\Fq^*)^4\,:\, \begin{array}{c}
                                            x_1^{b_{11}}y_1^{b_{12}} = x_2^{b_{11}}y_2^{b_{12}} \\
                                            x_1^{b_{21}}y_1^{b_{22}} = x_2^{b_{21}}y_2^{b_{22}} \\
                                            \rho^ix_1^{b_{11}}y_1^{b_{12}}+\rho^jx_1^{b_{21}}y_1^{b_{22}}=1
                                           \end{array} \right\},
\]

\[
\begin{aligned}
 \sum_{i,j}|\Eij|&=\left|\left\{(i,j,x_1,y_1,x_2,y_2)\in\ZZ_{q-1}^2\times(\Fq^*)^4\,:\, \begin{array}{c}
                                             x_1^{b_{11}}y_1^{b_{12}} = x_2^{b_{11}}y_2^{b_{12}} \\
                                            x_1^{b_{21}}y_1^{b_{22}} = x_2^{b_{21}}y_2^{b_{22}} \\
                                            \rho^ix_1^{b_{11}}y_1^{b_{12}}+\rho^jx_1^{b_{21}}y_1^{b_{22}}=1
                                            \end{array} \right\}\right| \\
 &= (q-2)\left|\left\{(x_1,y_1,x_2,y_2)\in(\Fq^*)^4\,:\, \begin{array}{c}
                                             x_1^{b_{11}}y_1^{b_{12}} = x_2^{b_{11}}y_2^{b_{12}} \\
                                            x_1^{b_{21}}y_1^{b_{22}} = x_2^{b_{21}}y_2^{b_{22}}
                                            \end{array} \right\}\right| \\
 &=(q-2)(q-1)^2\left|\left\{(x,y)\in(\Fq^*)^2\,:\, \begin{array}{c}
                                             x^{b_{11}}y^{b_{12}} = 1 \\
                                             x^{b_{21}}y^{b_{22}} = 1
                                            \end{array} \right\}\right| \\
 &= (q-2)(q-1)^2k.
\end{aligned}
\]
Now we have
\[
 \sum_{i,j}|\Cij^*|^2=\sum_{ij}|\Dij|+\sum_{ij}|\Eij|=(q-1)^2\bigg[(q-1)^2-(q-1)(d+e+f)+qk\bigg].
\]
Finally, using Eq.~\eqref{eqxxx}, we get $\sum_{i,j}N_{ij}^2=(q-1)^2q(k-d-e-f+2)$.
\end{proof}

\begin{prop}\label{lagrange}
 Let $G$ be an abelian group and let $g_1,g_2,g_3\in G$ such that $G=\langle g_1,g_2\rangle =
 \langle g_1,g_3\rangle =\langle g_2,g_3\rangle$. Let $K\geq 0$ and let $N:G\to\RR$ be a function
 $a \mapsto N_a:= N(a)$ such that
 \begin{enumerate}
  \item \label{prop1} $\displaystyle\sum_{a\in g+\langle g_1\rangle}N_a = 0$ for all $g\in G$,
  \item \label{prop2} $\displaystyle\sum_{a\in g+\langle g_2\rangle}N_a = 0$ for all $g\in G$,
  \item \label{prop4} $\displaystyle\sum_{a\in g+\langle g_3\rangle}N_a = 0$ for all $g\in G$,
  \item \label{prop3} $\displaystyle\sum_{a\in G}N_a^2=K$.
 \end{enumerate}
 Then
 \[
  |N_g| \leq \sqrt{K\left(1+\frac{2}{|G|}-\frac{1}{n_1}-\frac{1}{n_2}-\frac{1}{n_3}\right)},
 \]
 for all $g\in G$, where $n_1$, $n_2$, and $n_3$ are the orders of the elements $g_1$, $g_2$,
 and $g_3$, respectively.
\end{prop}
\begin{proof}
Let $n_{12}=|\langle g_1\rangle \cap \langle g_2\rangle|$. The isomorphism
\[
  G/\langle g_1\rangle\cong \langle g_2\rangle / \langle g_1\rangle \cap \langle g_2\rangle
\]
implies that $n_{12}=\frac{n_1n_2}{|G|}$. Similarly, we define $n_{13}$ and $n_{23}$.

We study first the case when $1+\frac2{|G|}=\frac1{n_1}+\frac1{n_2}+\frac1{n_3}$.
Assuming without loss of generality that $n_1 \leq n_2 \leq n_3$, the previous equality implies
that $n_1<3$. The case $n_1=1$ can only happen when $g_1$ is the neutral element of $G$, and then~\eqref{prop1}
reduces to $N_g=0$ for all $g\in G$. In the case $n_1=2$, we have $|G|=|\langle g_1,g_2\rangle|\leq 2n_2$,
hence $\frac{1}{2}+\frac{1}{n_2}+\frac{1}{n_3}=1+\frac{2}{|G|}\geq 1+\frac{1}{n_2}$, and in particular $n_3\leq 2$.
Therefore $n_1=n_2=n_3=2$ and $G$ is a group of order $|G|=4$. The elements $g_1,g_2,g_3$ are pairwise
distinct, since each pair of them generates the group, so $G = \{ 0, g_1, g_2, g_3 \} \simeq \ZZ_2\oplus\ZZ_2$. Items~\eqref{prop1},~\eqref{prop2},~\eqref{prop4} yield the following identities:
\[
\begin{aligned}
 N_{0}+N_{g_1}=&N_{g_2}+N_{g_3} = 0 \\
 N_{0}+N_{g_2}=&N_{g_1}+N_{g_3} = 0 \\
 N_{0}+N_{g_3}=&N_{g_1}+N_{g_2} = 0,
\end{aligned}
\]
which imply $N_g=0$ for all $g\in G$, and the claim follows.

Now we assume that $1+\frac2{|G|}\neq\frac1{n_1}+\frac1{n_2}+\frac1{n_3}$.
 We use Lagrange multipliers to get the desired upper bound for $N_g$. By the symmetry of the problem, we can restrict
 to the case $N_0$. Define the auxiliary function
 \[
 \begin{aligned}
  F = N_0 &+ \sum_{\bar{g}\in G/\langle g_1\rangle}\lambda_{\bar g}\left(\sum_{a\in g+\langle g_1\rangle}N_a\right)+
    \sum_{\bar{g}\in G/\langle g_2\rangle}\mu_{\bar g}\left(\sum_{a\in g+\langle g_2\rangle}N_a\right) \\
    & +\sum_{\bar{g}\in G/\langle g_3\rangle}\varepsilon_{\bar g}\left(\sum_{a\in g+\langle g_3\rangle}N_a\right)+
    \gamma\left(-K+\sum_{a\in G}N_a^2\right)
 \end{aligned}
 \]
with $N_a$, $\lambda_{\bar{g}}$, $\mu_{\bar{g}}$, $\varepsilon_{\bar{g}}$, and $\gamma$ as independent variables.
The critical points of $F$ correspond with the local extrema of $N_0$ subject to the restrictions stated in the theorem.
Now, we calculate the partial derivatives of $F$ with respect to each variable. With respect to $\lambda_{\bar{g}}$, $\mu_{\bar{g}}$,
$\varepsilon_{\bar{g}}$, and $\gamma$, we get the assumptions of the proposition. With respect to $N_a$, we get
\begin{equation}\label{eqyyy}
 \frac{\partial F}{\partial N_a} = \delta_{a,0} + \lambda_{\bar{a}} + \mu_{\bar{a}} + \varepsilon_{\bar{a}} + 2\gamma N_a = 0
\end{equation}
for all $a\in G$, where $\delta_{a,0}$ stands for the Kronecker delta.

For any element $g\in G$, we have
\[
 \begin{aligned}
 \sum_{a\in g+\langle g_1\rangle}(\delta_{a,0} + \lambda_{\bar{a}} &+ \mu_{\bar{a}} + \varepsilon_{\bar{a}} + 2\gamma N_a ) = \chi_{\langle g_1\rangle}(g)+
 n_1\lambda_{\bar{g}} + \sum_{a\in g+\langle g_1\rangle}\mu_{\bar{a}} + \sum_{a\in g+\langle g_1\rangle}\varepsilon_{\bar{a}} \\
 &= \chi_{\langle g_1\rangle}(g)+
 n_1\lambda_{\bar{g}} + n_{12}\sum_{\bar{a}\in G/\langle g_2\rangle}\mu_{\bar{a}}
 + n_{13}\sum_{\bar{a}\in G/\langle g_3\rangle}\varepsilon_{\bar{a}} =0.
 \end{aligned}
\]
Define $\lambda=-\frac{n_{12}}{n_1}\sum_{\bar{a}\in G/\langle g_2\rangle}\mu_{\bar{a}}-\frac{n_{13}}{n_1}\sum_{\bar{a}\in G/\langle g_3\rangle}\varepsilon_{\bar{a}}$. The previous identity shows that
$\lambda_{\bar{0}}=\lambda-\frac{1}{n_1}$ and $\lambda_{\bar{g}}=\lambda$ for all $\bar{g}\neq\bar{0}$. Similarly, we define
$\mu=-\frac{n_{12}}{n_2}\sum_{\bar{a}\in G/\langle g_1\rangle}\lambda_{\bar{a}}-\frac{n_{23}}{n_2}\sum_{\bar{a}\in G/\langle g_3\rangle}\varepsilon_{\bar{a}}$, and then $\mu_{\bar{0}}=\mu-\frac{1}{n_2}$ and $\mu_{\bar{g}}=\mu$ for all $\bar{g}\neq\bar{0}$.
Analogously, we define $\varepsilon=-\frac{n_{13}}{n_3}\sum_{\bar{a}\in G/\langle g_1\rangle}\lambda_{\bar{a}}-\frac{n_{23}}{n_3}\sum_{\bar{a}\in G/\langle g_2\rangle}\mu_{\bar{a}}$, and then $\varepsilon_{\bar{0}}=\varepsilon-\frac{1}{n_3}$ and $\varepsilon_{\bar{g}}=\varepsilon$ for all $\bar{g}\neq\bar{0}$. By construction of $\varepsilon$, we have
\[
  \begin{aligned}
  \varepsilon &= -\frac{n_{13}}{n_3}\sum_{\bar{a}\in G/\langle g_1\rangle}\lambda_{\bar{a}}-\frac{n_{23}}{n_3}\sum_{\bar{a}\in G/\langle    g_2\rangle}\mu_{\bar{a}} \\
  &= -\frac{n_{13}}{n_3}\left(\frac{|G|}{n_1}\lambda-\frac1{n_1}\right)-\frac{n_{23}}{n_3}\left(\frac{|G|}{n_2}\mu-\frac1{n_2}\right).
  \end{aligned}
\]
Therefore $\lambda+\mu+\varepsilon=\frac{2}{|G|}$, and Equation~\eqref{eqyyy} can be rewritten as follows:
\begin{equation}\label{eqzzz}
 2\gamma N_a = -\delta_{a,0} - \frac{2}{|G|} + \frac{\chi_{\langle g_1\rangle}(a)}{n_1} + \frac{\chi_{\langle g_2\rangle}(a)}{n_2}
 + \frac{\chi_{\langle g_3\rangle}(a)}{n_3}.
\end{equation}
Squaring the previous equation and summing over all $a\in G$, we get
\[
 4\gamma^2\sum_{a\in G}N_a^2 = 1+\frac2{|G|}-\frac1{n_1}-\frac1{n_2}-\frac1{n_3}\neq 0.
\]
This allows us to get $\gamma\neq 0$, and together with Equation~\eqref{eqzzz} for $a=0$, concludes the proof.
\end{proof}

\begin{proof}[Proof of Corollary~\ref{maincor}]
 Consider $G=\ZZ_{q-1}^2/\langle (b_{11},b_{12}), (b_{21},b_{22})\rangle=\coker(B)$. By Theorem~\ref{mainth}\eqref{x4}, the map
 $N:G\to\RR$ such that $\overline{(i,j)}\mapsto N_{ij}$ is well defined. Let $g_1=\overline{(1,0)}\in G$, $g_2=\overline{(0,1)}\in G$,
 and $g_3=\overline{(1,1)}\in G$. With this notation, the hypotheses of Proposition~\ref{lagrange} with $K=kq(k-d-e-f+2)$ follow from
 Theorem~\ref{mainth} and Eq.~\eqref{bound1}.
 By Lemma~\ref{orden}, we have $n_1=k/e$, $n_2=k/d$, and $n_3=k/f$. The only thing left to do is to substitute these values in
 Proposition~\ref{lagrange} and a suitable rearrangement of the terms.
\end{proof}

\section{Gauss' theorem}\label{gauss}

We devote this section entirely to showing how to derive Thm.~\ref{thm_gauss} as a consequence of Thm.~\ref{mainth}. We consider
the family of curves
\[
 \Cij = \{ (x,y)\in\Fp^2\,:\,\rho^ix^3+\rho^jy^3=1 \}
\]
for $0\leq i,j<p-1$.
Removing the extra points on the lines $x=0$, $y=0$, $z=0$, Gauss' curve corresponds to $\{ (x,y)\in(\Fp^*)^2\,:\,x^3+y^3=-1 \}$
that has the same number of points as $\mathcal{C}^*_{00}$. The number of points on each of those lines is equal to the number
of cubic roots of the unity in $\Fp$, which is equal to $\gcd(3,p-1)$ by Lemma~\ref{lemaD2}. Therefore, $M_p=|\mathcal{C}_{00}^*|+3\gcd(3,p-1)$. By Thm.~\ref{mainth}, we get
\[
 M_p=p+1-D_d(0)-D_e(0)-D_f(w)+N_{00}+3\gcd(3,p-1), 
\]
where $d=e=f=\gcd(3,p-1)$ and $w=\frac{p-1}{2}$. Then $M_p=p+1+N_{00}$.

\bigskip

Case $p\not\equiv 1\pmod{3}$: Here we have $d=e=f=1$, $k=\gcd(p-1, 9)=1$, then $\tilde{g}=0$ and $N_{ij}=0$ for all $0\leq i,j<p-1$
by Thm.~\ref{mainth}\eqref{x3}. In particular $N_{00}=0$ and $M_p=p+1$, as expected.

\bigskip

Case $p\equiv 1\pmod{3}$: Here we have $d=e=f=3$, $k=\gcd(3(p-1),9)=9$ and $\tilde{g}=\frac{1}{2}(k-d-e-f+2)=1$. The cokernel of
the matrix $B=\begin{bmatrix} 3 & 0 \\ 0 & 3\end{bmatrix}$ is $\ZZ_3\oplus\ZZ_3$, so the numbers $N_{ij}$ reduce to only
nine possibilities
\[
A=
 \begin{bmatrix}
  N_{00} & N_{01} & N_{02}\\
  N_{10} & N_{11} & N_{12}\\
  N_{20} & N_{21} & N_{22}
 \end{bmatrix}
\]
depending on the class of $(i,j)$ in $\coker(B)$ by Thm.~\ref{mainth}\eqref{x4}. Due to Thm.~\ref{mainth}\eqref{x2}\eqref{x1}\eqref{x5},
each of the rows, columns and diagonals of the matrix $A$ above adds up to zero. This proves that
\[
A=
 \begin{bmatrix}
  u & v & -u-v\\
  v & -u-v & u\\
  -u-v & u & v
 \end{bmatrix}
\]
for some $u,v\in\ZZ$. Moreover, by Eq.~\eqref{bound1}, the sum of the squares of the entries of $A$ is
$3(u^2+v^2+(u+v)^2)=18p$, so
\begin{equation}\label{eq_uv}
u^2+v^2+uv=3p. 
\end{equation}
Let $\xi_3\in \Fp$ be a cubic root of the unity. Note that $|\mathcal{C}^*_{ij}|$ is divisible by $9$, since for each point $(x,y)\in\mathcal{C}^*_{ij}$, its conjugates $(\xi_3^rx,\xi_3^sy)$ are also in $\mathcal{C}^*_{ij}$ for any $0\leq r,s<3$.
By Eq.~\eqref{eq:defCij},
$$
\begin{aligned}
u &= N_{00} = |\mathcal{C}_{00}^{*}| - (p+1) + D_3(0) + D_3(0) + D_3(w) =  |\mathcal{C}_{00}^{*}| - p + 8, \\
v &= N_{01} = |\mathcal{C}_{01}^{*}| - (p+1) + D_3(0) + D_3(1) + D_3(w-1) = |\mathcal{C}_{01}^{*}| - p + 2.
\end{aligned}
$$
Therefore $2v + v = 2 |\mathcal{C}_{01}^{*}| + |\mathcal{C}_{00}^{*}| - 3(p-4)$ is divisible by $9$.
Denoting $\bar{v}=\frac{2v+u}{9}\in\ZZ$, Eq.~\eqref{eq_uv} becomes $u^2+27\bar{v}^2=4p$.
Moreover, $u = |\mathcal{C}_{00}^{*}| - p + 8 \equiv 1 \pmod{3}$.
The uniqueness of the solution of $u^2+27\bar{v}^2=4p$ with $u\equiv1\pmod{3}$ follows from the fact that $\ZZ\left[\frac{-1+\sqrt{-3}}{2}\right]$ is a UFD.

\section{Genus of a trinomial curve}\label{sec4}

The aim of this section is to calculate the genus of the projective closure $\overline{\Cij}$ of the curve
$\Cij$ in $\mathbb{P}^2(\overline{\Fq})$ in the irreducible case. In order to do so, we use the standard
formula that relates the genus of a curve with the delta invariant $\delta_P$ at each of its
singularities, see formula~\eqref{www} below. The delta invariants are computed using the techniques shown
in~\cite[Ch.~3 and 6]{Casas}. The final formula obtained in Prop.~\ref{prop-genus} should be compared
term by term to the definition of $\tilde{g}$ given in Eq.~\eqref{eqdef}.

\begin{lemma}\label{deltax}
Let $\mathcal{C}\subseteq \PP^2(\overline{\Fq})$ be a curve such that its local equation at $P$ is given by
$\alpha x^r+\beta y^s+\gamma x^uy^v=0$ with $\alpha\beta\neq 0$ and $r,s\geq 1$. If either $\gamma=0$ or $(u,v)$ is above the
segment that joins $(r,0)$ and $(0,s)$, then
  \[  \delta_{P}=\frac12\left(rs-r-s+\gcd(r,s)\right).  \]
If $\gamma\neq 0$ and $(u,v)$ is below the segment, then
  \[  \delta_{P}=\frac12\left(rv+su-r-s+\gcd(u,s-v)+\gcd(v,r-u)\right). \]
The formula of the first case is valid even if $r=0$ or $s=0$. Also, the formula of the second
case is valid when either $r=u=0$ or $s=v=0$. In both situations, the point $P$ does not belong
to the curve and $\delta_P=0$.
\end{lemma}
\begin{proof}
 In the first case, the term $\gamma x^uy^v$ can be removed from the local equation without changing the topology (since the point
 is above the Newton polygon). It is clear that the Milnor number at $P$ is $\mu_P=(r-1)(s-1)$ and that the number of local branches at $P$
 is $r_P=\gcd(r,s)$. Therefore $2\delta_P=\mu_P+r_P-1=rs-r-s+\gcd(r,s)$.
 
 In the other case, the local equation can be changed by $\alpha x^r + \beta y^s +\gamma x^uy^v + \frac{\alpha\beta}{\gamma}x^{r-u}y^{s-v}$,
 since the extra term is above the Newton polygon. Doing so, we get an expression that factorizes as $(\alpha x^u + \frac{\alpha\beta}{\gamma}y^{s-v})(x^{r-u}+\frac\gamma\alpha y^v)$. Applying the formula of the $\delta$-invariant of a product, we get:
\[
\delta_P(\alpha x^r+\beta y^s+\gamma x^uy^v) = \delta_P(\alpha x^u + \frac{\alpha\beta}{\gamma}y^{s-v}) + \delta_P(x^{r-u}+\frac\gamma\alpha y^v) + i_P(\alpha x^u + \frac{\alpha\beta}{\gamma}y^{s-v},x^{r-u}+\frac\gamma\alpha y^v),
\]
 where $i_P$ denotes the intersection multiplicity at $P$. The values of $\delta_P$ of each factor can be computed as in the first case.
 Using Noether's formula (see~\cite[p.~3568]{CMO}), the intersection multiplicity is $uv$. We conclude by simply adding these three values.
\end{proof}

\begin{prop}\label{prop-genus}
 If the projective closure $\overline{\Cij}$ of the curve $\Cij$ is irreducible in $\PP^2(\overline{\Fq})$, the genus of $\overline{\Cij}$ is
 \[
  g(\overline{\Cij}) = \frac12\left(|\det(B)|-\gcd(b_{11}, b_{12})-\gcd(b_{21},b_{22})-\gcd(b_{11}-b_{21},b_{12}-b_{22})+2 \right).
 \]
\end{prop}
\begin{proof}
By using the irreducibility of the curve, we can reduce the proof to the case $a_{12}=a_{21}=0$ and $a_{11}\geq a_{22}$.
In this case, the genus can be computed using the following formula:
\begin{equation}\label{www}
 g(\overline{\Cij})=\frac{(m-1)(m-2)}{2}-\sum_{P}\delta_P,
\end{equation}
where $m$ is the degree of the curve, $P$ ranges over all singular points of $\overline{\Cij}$, and $\delta_P$ is the
$\delta$-invariant of $\overline{\Cij}$ at $P$. Since we are assuming $\det(B)\neq 0$, the set of singular points is
contained in $\{[1:0:0], [0:1:0], [0:0:1]\}$.

We have to consider two cases: (1) $m=a_{11}>a_{31}+a_{32}$, (2) $m=a_{31}+a_{32}\geq a_{11}$.

Case (1): The projective closure of $\Cij$ is given by the homogeneous polynomial
\[
  F(x,y,z)=\rho^ix^{a_{11}}+\rho^jy^{a_{22}}z^{a_{11}-a_{22}}-x^{a_{31}}y^{a_{32}}z^{a_{11}-a_{31}-a_{32}}. 
\]
We can further assume without loss of generality that the point $(a_{31},a_{32})$ is below the line that
connects the points $(a_{11},0)$ and $(0,a_{22})$, i.e. the Newton polygon of $F$ has two edges. Otherwise,
we would simply exchange $y$ by $z$ and start over. Note that in the exceptional case when $a_{11}=0$ ($a_{22}=0$), we should
also have $a_{31}=0$ ($a_{32}=0$), respectively. The singular points are $[0:1:0]$ and $[0:0:1]$, and the
local equations of $\overline{\Cij}$ at those points are $\rho^ix^{a_{11}}+\rho^jz^{a_{11}-a_{22}}-x^{a_{31}}z^{a_{11}-a_{31}-a_{32}}$
and $\rho^ix^{a_{11}}+\rho^jy^{a_{22}}-x^{a_{31}}y^{a_{32}}$, respectively. By Lemma~\ref{deltax}, we have
\[
 \begin{aligned}
   \delta_{[0:1:0]} &= \frac12 \left(a_{11}(a_{11}-a_{22})-a_{11}-(a_{11}-a_{22})+\gcd(a_{11},a_{11}-a_{22})\right),\\
   \delta_{[0:0:1]} &= \frac12 \left(a_{11}a_{32}+a_{22}a_{31}-a_{11}-a_{22}+\gcd(a_{31}, a_{22}-a_{32})+\gcd(a_{32}, a_{11}-a_{31})\right).
 \end{aligned}
\]
Finally, using~\eqref{www}, we get the desired formula.

Case (2): The projective closure of $\Cij$ is given by the homogeneous polynomial
\[
  F(x,y,z)=\rho^ix^{a_{11}}z^{a_{31}+a_{32}-a_{11}}+\rho^jy^{a_{22}}z^{a_{31}+a_{32}-a_{22}}-x^{a_{31}}y^{a_{32}}. 
\]
The local equations of $\overline{\Cij}$ at $[1:0:0]$, $[0:1:0]$, $[0:0:1]$ are $\rho^iz^{a_{31}+a_{32}-a_{11}}+\rho^jy^{a_{22}}z^{a_{31}+a_{32}-a_{22}}-y^{a_{32}}$, $\rho^ix^{a_{11}}z^{a_{31}+a_{32}-a_{11}}+\rho^jz^{a_{31}+a_{32}-a_{22}}-x^{a_{31}}$, and $\rho^ix^{a_{11}}+\rho^jy^{a_{22}}-x^{a_{31}}y^{a_{32}}$, respectively. By Lemma~\ref{deltax},
\[
 \begin{aligned}
   \delta_{[1:0:0]} &= \frac12 \left(a_{32}(a_{31}+a_{32}-a_{11})-a_{32}-(a_{31}+a_{32}-a_{11})+\gcd(a_{32},a_{31}-a_{11})\right),\\
   \delta_{[0:1:0]} &= \frac12 \left(a_{31}(a_{31}+a_{32}-a_{22})-a_{31}-(a_{31}+a_{32}-a_{22})+\gcd(a_{31},a_{32}-a_{22})\right),\\
   \delta_{[0:0:1]} &= \frac12 \left(a_{11}a_{22}-a_{11}-a_{22}+\gcd(a_{11}, a_{22})\right).
 \end{aligned}
\]
The conclusion follows from formula~\eqref{www}.
\end{proof}

\section{Examples}

We study some particular cases of Theorem~\ref{mainth} and Corollary~\ref{maincor} which have
special significance by themselves.

\begin{ex}[Diagonal case]\label{example1}
 The curve $\Cij=\{(x,y)\in(\Fq^*)^2\,:\,\rho^ix^{a_{11}}+\rho^jy^{a_{22}}=1\}$ has $d=\gcd(a_{11},q-1)$, $e=\gcd(a_{22},q-1)$,
 $f=\gcd(d,e)$, and $w=(q-1)/2$ for $q$ odd, and $w=0$ otherwise. Moreover,
 \[
   \coker(B)=\ZZ_{q-1}^2/\langle(a_{11},0),(0,a_{22})\rangle\cong \ZZ_d\oplus\ZZ_e,
 \]
  hence $k=|\coker(B)|=de$. By Theorem~\ref{mainth}\eqref{x4}, we have $N_{i+d,j}=N_{ij}=N_{i,j+e}$, so the matrix
  $[N_{ij}]_{0\leq i,j<q-1}$ has its upper-left block of size $d\times e$ repeated $(q-1)^2/de$ times. As multisets,
  we can write:
  \[
    \{ N_{ij}\,:\, 0\leq i,j<q-1\} = \frac{(q-1)^2}{de}\cdot  \{ N_{ij}\,:\, 0\leq i<d,\, 0\leq j<e\}.
  \]
  Moreover, the sum in Thm~\ref{mainth}\eqref{x2} can be taken from $j=0$ to $j=e-1$. Similarly, the sum~\eqref{x1}
  can be taken from $i=0$ to $i=d-1$.
\end{ex}  

\begin{ex}\label{example2}
  We consider the subcase of Example~\ref{example1} when $a_{11}$ is odd, $a_{22}=2$, and $q$ is odd.
  The constants $d$, $e$, $f$, $k$, $w$ reduce to $d=\gcd(a_{11},q-1)$, $e=2$, $f=1$, $k=2d$, $w=(q-1)/2$, and
  the upper-left block is of size $d\times 2$. Since the second column of this block is the additive inverse of the first
  one, we have, as multisets:
  \[
    \{ N_{ij}\,:\, 0\leq i,j<q-1\} = \frac{(q-1)^2}{2d}\cdot \{ \pm \alpha_0, \pm \alpha_1, \ldots, \pm \alpha_{d-1}\}
  \]
  where $\alpha_i=N_{i0}$. Moreover,
  \[
    \begin{aligned}
     \alpha_0+\alpha_1+\cdots+\alpha_{d-1} &= 0,\\
     \alpha_0^2+\alpha_1^2+\cdots+\alpha_{d-1}^2 &= d(d-1)q.
    \end{aligned}
  \]
  The vector $(\alpha_0,\ldots,\alpha_{d-1})$ is in the intersection of a sphere and a hyperplane in $\RR^d$, i.e.
  the vector $(\alpha_1,\ldots,\alpha_{d-1})$ belongs to a conic in $\RR^{d-1}$. Of course, when $d=1$, the sphere
  reduces to a point.
\end{ex}  

\begin{ex}\label{example3}
 Now, we consider the curves $\Cij=\{(x,y)\in(\Fq^*)^2\,:\,\rho^ix^3+\rho^jy^2=1\}$, which is a particular case of
 the previous example. Clearly, when $q\not\equiv 1\pmod{3}$, we have $d=1$ and all the $N_{ij}$ are zero. For this
 reason, we only consider $q\equiv 1\pmod{3}$, in which case $d=3$:
 \[
  \{ N_{ij}\,:\, 0\leq i,j<q-1\} = \frac{(q-1)^2}{6}\cdot \{ \pm \alpha_0, \pm \alpha_1, \pm \alpha_2\},
 \]
 where $\alpha_0+\alpha_1+\alpha_2=0$ and $\alpha_0^2+\alpha_1^2+\alpha_2^2=6q$.
 \begin{itemize}
  \item If $q=p^{2n}$ for some $p\equiv 2\pmod{3}$, then $\alpha_0=p^n\beta_0$, $\alpha_1=p^n\beta_1$, and $\alpha_2=p^n\beta_2$,
  for some $\beta_0,\beta_1,\beta_2\in\ZZ$ such that $\beta_0+\beta_1+\beta_2=0$ and $\beta_0^2+\beta_1^2+\beta_2^2=6$.
  This implies that, as multisets:
  \[
   \{ N_{ij}\,:\, 0\leq i,j<q-1\} = \frac{(q-1)^2}{6}\cdot \{ \pm p^n, \pm p^n, \mp 2p^n\}.
  \]
  In particular, $N_{ij}\neq 0$ for all $i,j$ and the upper bound of Theorem~\ref{mainth} is sharp for this family.
  \item In constrast, for $p\equiv 1\pmod{3}$, the upper bound is not sharp. For instance, when $q=p=997$, we have
  $\alpha_0=10$, $\alpha_1=49$, $\alpha_2=-59$, but the integer part of the upper bound is $\lfloor(k-d-e-f+2)\sqrt{q}\rfloor=63$.
  Note, however, that
  \begin{equation}\label{opti}
    \max\left\{ x: \begin{array}{c}x,y,z\in\ZZ\\ x+y+z=0\\ x^2+y^2+z^2=6q\end{array}\right\} = 59,
  \end{equation}
  so one may think that the largest $N_{ij}$ can be obtained always by solving optimization problem in Prop.~\ref{lagrange}
  for a function $f:G\to\ZZ$.
  \item In the case $q=7^2$, we have $\{ N_{ij}\,:\, 0\leq i,j<48\} = 384 \cdot \{ \mp 2, \mp 11, \pm 13\}$, so the largest
  $N_{ij}$ is $13$. However, the integer optimization problem~\eqref{opti} gives $14$. This highlights the fact that the relations given in Thm~\ref{mainth} are not always enough to characterize the maximum $N_{ij}$.
 \end{itemize}
\end{ex}

\begin{ex}\label{example4} 
  When $a_{22}=2$, $q$ is odd, and $a_{11}$ is even, the situation is similar, but $f=2$, and item~\eqref{x5} of
  Thm~\ref{mainth}, gives the additional relation $\alpha_0-\alpha_1+\cdots+\alpha_{d-2}-\alpha_{d-1}=0$. All together
  this gives
  \[
    \begin{aligned}
     \alpha_0+\alpha_2+\cdots+\alpha_{d-2} &= 0,\\
     \alpha_1+\alpha_3+\cdots+\alpha_{d-1} &= 0,\\
     \alpha_0^2+\alpha_1^2+\cdots+\alpha_{d-1}^2 &= d(d-2)q.
    \end{aligned}
  \]
\end{ex}

\begin{ex}\label{example5} 
 The curve $\Cij=\{(x,y)\in(\Fq^*)^2\,:\,\rho^ix^3+\rho^jy^2=x\}$ with odd $q$, has $d=2$, $e=f=1$, $k=\gcd(q-1,4)$, and
 $w=(q-1)/2$. When $q\equiv 3\pmod{4}$, we have $k=2$, so $k-d-e-f+2=0$, and in particular $N_{ij}=0$ for all $i,j$. The
 other case, i.e. $q\equiv 1\pmod{4}$, is more interesting. Here $k=4$, and
 \[
   \coker(B)=\ZZ_{q-1}^2/\langle(2,0),(-1,2)\rangle = \{ \overline{(0,0)}, \overline{(1,0)}, \overline{(0,1)}, \overline{(1,1)} \}.
 \]
By Thm~\ref{mainth}\eqref{x4}, $N_{i+2,j}=N_{ij}=N_{i-1,j+2}$, so as multisets
\[
 \{ N_{ij}\,:\, 0\leq i,j<q-1\} = \frac{(q-1)^2}{4}\cdot\{\pm \alpha, \pm\beta\},
\]
where $\alpha=N_{00}$, $\beta=N_{01}$, and $\alpha^2+\beta^2=4q$.
If, we also have $p\equiv 3\pmod{4}$, then $q=p^{2n}$ and the multiset is $\frac{(q-1)^2}{4}\cdot\{\pm 2p^n, 0\}$.
\end{ex}

\Addresses

\end{document}